\documentclass[12pt,a4paper,oneside]{amsart}
\usepackage[utf8]{inputenc}
\usepackage{amsmath,amscd,hyperref, amsfonts, amssymb, amsthm}
\usepackage{comment}
\usepackage{enumitem}
\usepackage{setspace,kantlipsum}
\usepackage{tikz-cd}
\usepackage{mathrsfs}
\usepackage{textcomp}
\usepackage[margin=1in]{geometry}
\usepackage{colonequals}
\usepackage{mathtools}
\DeclarePairedDelimiter{\abs}{\lvert}{\rvert}
\theoremstyle{plain} 
\newtheorem{thm}{Theorem}[section]
\newtheorem{lemma}[thm]{Lemma}
\newtheorem{proposition}[thm]{Proposition}
\newtheorem{corollary}[thm]{Corollary}

\theoremstyle{definition}

\newtheorem{definition}[thm]{Definition}

\newtheorem{remark}[thm]{Remark}

\newtheorem{conjecture}[thm]{Conjecture}

\newcommand{\C}{\mathbb{C}}
\newcommand{\Z}{\mathbb{Z}}
\newcommand{\Q}{\mathbb{Q}}
\newcommand{\R}{\mathbb{R}}
\newcommand{\N}{\mathbb{N}}
\newcommand{\K}{\mathbb{K}}

\newcommand{\cC}{\mathcal{C}}

\newcommand{\cF}{\mathcal{F}}

\newcommand{\cH}{\mathcal{H}}

\newcommand{\cP}{\mathcal{P}}

\newcommand{\fg}{\mathfrak{g}}

\newcommand{\Alb}{\mathrm{Alb}}

\newcounter{tmp}
\begin{document}

\title{Zeros of one-forms and homologically trivial fibrations}

\author{Stefan Schreieder and Ruijie Yang}

\maketitle


 \begin{abstract}  
We show that a conjecture of Kotschick about one-forms without zeros on compact K\"ahler manifolds follows in the case of simple Albanese torus from a conjecture of Bobadilla and Koll\'ar about homologically trivial fibrations. 
As an application, we prove Kotschick's conjecture for compact K\"ahler manifolds $X$ with $b_1(X)\geq 2 \dim X-2$, whose Albanese torus is simple.
\end{abstract}


%

\section{Introduction}

In \cite{Kotschick}, Kotschick made the following
\begin{conjecture}\label{conjecture: Kotschick}
For a compact K\"ahler manifold $X$, the following conditions are equivalent:
\begin{itemize}
    \item[(A)] $X$ admits a holomorphic one-form without zeros;
    \item[(B)]$X$ admits a real closed one-form without zeros; or by Tischler's theorem \cite{Tischler} equivalently, the underlying differential manifold of $X$ is a $\cC^{\infty}$-fiber bundle over the circle.
\end{itemize}
\end{conjecture}

We propose the following stronger conjecture, which implies Conjecture \ref{conjecture: Kotschick} when we apply it to the Albanese morphism $X\to \Alb(X)$.
\begin{conjecture}\label{conjecture: Kotschick stronger}
Let $X$ be a compact K\"ahler manifold and let $f:X\to A$ be a morphism to a complex torus $A$. Then the following conditions are equivalent:
\begin{enumerate}[label=(\Alph*)]
    \item\label{condition A}$X$ admits a holomorphic one-form $w$ without zeros such that 
    \[ [w]\in f^{\ast}H^0(A,\Omega^1_A).\] 
     \item \label{condition B} $X$ admits a real closed one-form $\alpha$ without zeros such that $[\alpha]\in f^\ast H^1(A,\mathbb R)$; or, by Tischler's argument \cite{Tischler} equivalently, 
     the underlying differential manifold of $X$ is a $\cC^{\infty}$-fiber bundle $g:X\to S^1$ over the circle with 
     \[g^\ast H^1(S^1,\mathbb R)\subseteq f^\ast H^1(A,\mathbb R).\] 
\end{enumerate}
\end{conjecture}

In both conjectures, it is easy to see that Condition (A) implies Condition (B). The converse direction is the non-trivial part. 
These conjectures are known for surfaces \cite{Kotschick,Sch-IMRN} and for projective threefolds \cite{HS}.
(In \textit{loc.\ cit.}, Conjecture \ref{conjecture: Kotschick} is considered, but the arguments prove in fact the slightly stronger assertion from Conjecture \ref{conjecture: Kotschick stronger}, cf.\ \cite[Theorem 1.4]{HS}.)

The observation of this paper is that we can relate Conjecture \ref{conjecture: Kotschick stronger} to the following conjecture of Bobadilla and Koll\'ar \cite[Conjecture 3]{Bobadilla}. For a commutative ring $R$, a proper morphism $f:X\to Y$ between complex analytic spaces is called an $R$-homology fiber bundle if $Y$ has an open cover $Y=\cup_{i} U_i$ such that for every $i$ and every $y\in U_i$, the map induced by inclusion
\[ H_{\ast}(f^{-1}(y),R) \to H_{\ast}(f^{-1}(U_i),R) \quad \textrm{is an isomorphism}.\]
\begin{conjecture}\label{conjecture: BK}
    Let $f:X\to Y$ be a proper morphism between complex analytic spaces, where $X$ and $Y$ are both smooth. If $f$ is a $\Z$-homology fiber bundle, then $f$ is smooth.
\end{conjecture}

Our main result is
\begingroup
\setcounter{tmp}{\value{thm}}
\setcounter{thm}{0} 
\renewcommand\thethm{\Alph{thm}}
\begin{thm}\label{thm: main}
Let $X$ be a compact K\"ahler manifold and let $f:X\to A$ be a morphism to a simple complex torus $A$. Assume that Conjecture \ref{conjecture: BK} holds for the morphism $f$, then Conjecture \ref{conjecture: Kotschick stronger} holds for $f$.
\end{thm}

To prove the theorem above, we show that if $f:X\to A$ is a morphism to a simple complex torus $A$ such that there is a closed real 1-form $\alpha$ on $X$ without zeros and such that $[\alpha]\in f^\ast H^1(A,\R)$, then $f$ is a $\mathbb{Z}$-homology fiber bundle, see Proposition \ref{prop: vanishing euler characteristic for direct image of Albanese} and Corollary \ref{cor: Albanese morphism is a Z homology fiber bundle}.
This generalizes a recent result of Dutta--Hao--Liu \cite[Corollary 1.6]{DHL}, who proved that $f$ is a $\mathbb{C}$-homology fiber bundle (under the assumption that $X$ is projective). 
In order to obtain the integral statement, one essentially has to prove that  $f$ is a $\K$-homology fiber bundle for any infinite field $\K$.
To this end we use a different method than in \cite{DHL}: instead of the Kashiwara estimate for $\C$-perverse sheaves, we use 
the generic vanishing theorem for $\K$-perverse sheaves by Bhatt--Schnell--Scholze and a result of Kr\"amer and Weissauer on classification of $\K$-perverse sheaves with vanishing Euler characteristics. 

In the case of non-simple tori, Conjectures \ref{conjecture: Kotschick stronger} does not directly follow from \ref{conjecture: BK}.
Indeed, there are smooth complex projective threefolds $X$ such that for any morphism $f:X\to A$ to a positive-dimensional complex torus $A$, $f$ is not even a $\mathbb Q$-homology fiber bundle (e.g.\ this happens for the blow-up of $E_1\times E_2\times \mathbb P^1$ along the union of $ E_1\times 0\times 0$ and $0\times E_2\times\infty$, where $E_1,E_2$ denote non-isogeneous elliptic curves).

The assumption that $A$ is simple is automatic in the case where $A$ is an elliptic curve and so the above theorem gives good evidence that Conjecture \ref{conjecture: Kotschick} may hold in the case where $b_1(X)=2$; an interesting wide open special case.

It is straightforward to see that Conjecture \ref{conjecture: BK} holds for any proper morphism of relative dimension $\leq 0$.
By \cite[Proposition 10]{Bobadilla}, Conjecture \ref{conjecture: BK}  also holds for proper morphisms of relative dimension 1. 
Therefore we have 
\begin{corollary}\label{corollary: strong Kotschick conjecture}
Let $X$ be a compact K\"ahler manifold and let $f:X\to A$ be a morphism to a simple complex torus $A$.
Assume that $\dim A\geq \dim X-1$.
Then Conjecture \ref{conjecture: Kotschick stronger} holds for $f$.
\end{corollary}

\begin{corollary}\label{corollary: Kotschick conjecture large betti number}
Let $X$ be a compact K\"ahler manifold such that $\Alb(X)$ is simple and $b_1(X)\geq 2\dim X-2$. Then Conjecture \ref{conjecture: Kotschick} holds for $X$.
\end{corollary}

By a result of Popa and Schnell \cite{PS}, smooth projective varieties of general type do not admit nowhere vanishing holomorphic one-forms. 
If the Albanese variety is simple, this had earlier been proven in \cite[Theorem 1.4]{HK}. 
Therefore we have
\begin{corollary}
    Let $X$ be a smooth projective variety of general type such that $\Alb(X)$ is simple and $b_1(X)\geq 2\dim X-2$. Then the underlying differential manifold of $X$ cannot be a $\cC^{\infty}$-fiber bundle over the circle.  
\end{corollary}
\endgroup

\subsection*{Acknowledgement}
We would like to thank Thomas Kr\"amer and Botong Wang for helpful discussion, a special thank to Botong Wang for pointing out the reference \cite{MS} and the first arxiv version of \cite{LMWBobadilla}.
We are grateful to the excellent referee for comments that improved the exposition.
SS is supported by the European Research Council (ERC) under the European Union's Horizon 2020 research and innovation programme under grant agreement No 948066. 
RY is supported by a Riemann Fellowship of Riemann Center for Geometry and Physics at Leibniz Universit\"at Hannover during the preparation of this paper.

\section{Perverse sheaves on abelian varieties}

In this section, let us review some results of Bhatt-Schnell-Scholze \cite{BSS} and Kr\"amer-Weissauer \cite{KW} about perverse sheaves on abelian varieties.

Let $A$ be a compact complex torus and let $\K$ be any field. 

\begin{thm}{\cite[Theorem 1.1]{BSS}}\label{thm: generic vanishing BSS}
    Let $\cP$ be a $\K$-perverse sheaf on $A$. Then for a generic rank one $\K$-local system $L$ on $A$, we have 
        \[ H^i(A,\cP\otimes_{\K}L)=0, \quad \textrm{for all } i\neq 0.\] 
\end{thm}

\begin{definition}
     For any $\K$-perverse sheaf $\cP$ on $A$, the \emph{Euler characteristic} of $\cP$ is defined by
    \[ \chi(\cP)\colonequals \sum_{i} (-1)^i \dim_{\K} H^i(A,\cP).\]
\end{definition}

 Under the additional assumption that $A$ is simple, we have the following important result.

\begin{thm}[Kr\"amer-Weissauer]\label{thm: locally free criterion for perverse sheaves}
    Let $\cP$ be a simple $\K$-perverse sheaf on $A$. Suppose that $\chi(\cP)=0$ and $A$ is simple, then $\cP$ is a shift of a local system.

\end{thm}
 
If $A$ is a simple abelian variety, the above theorem follows from \cite[Proposition 10.1]{KW}, cf.\ \cite[Proposition 5.11]{LMWBobadilla} (for more details see Corollary 5.15 in the first arXiv version of \cite{LMWBobadilla}). 
The argument extends to the case of complex tori and we include some details in the appendix of this paper, see Proposition \ref{Proposition A} below.
 
\begin{corollary}\label{corollary: locally free criterion for non-simple perverse sheaf}
 Let $\cP$ be a $\K$-perverse sheaf on $A$. Suppose that $\chi(\cP)=0$ and $A$ is simple, then $\cP$ is a shift of a local system.%
\end{corollary}

\begin{proof}
First, Theorem \ref{thm: generic vanishing BSS} implies that the Euler characteristic of any perverse sheaf $\cP$ is non-negative: let $L$ be a generic rank one local system, then
\[ \chi(\cP)=\chi(\cP\otimes L)=\dim H^0(A,\cP\otimes L)\geq 0.\]
The first equality comes from the invariance of Euler characteristic under twisting by a rank one local system and the second equality uses Theorem \ref{thm: generic vanishing BSS}. Now suppose $\chi(\cP)=0$. We can write $\cP$ as a successive extension of simple perverse sheaves $\cP_{\alpha}$. Since the Euler characteristic is additive in short exact sequences and Euler characteristic for any perverse sheaf $\cP_{\alpha}$ is always non-negative, we see that $\chi(\cP_\alpha)=0$. By Theorem \ref{thm: locally free criterion for perverse sheaves}, we know that $\cP_{\alpha}$ is a shift of a local system by the same constant $\dim X$. Therefore we conclude that $\cP$ is a shift of a local system. 
\end{proof} 

\section{The proof}

In this section, we prove Theorem \ref{thm: main} and deduce its corollaries. 

First, let us study the topological implication of the existence of a nowhere vanishing real one-form. It exploits the interplay of two structures: a $\cC^{\infty}$-fiber bundle structure over $S^1$ and a morphism to a compact torus.

\begin{proposition}\label{prop: vanishing euler characteristic for direct image of Albanese}
Let $f:X\to A$ be a morphism from a compact K\"ahler manifold $X$ to a simple complex torus $A$.
Suppose that the condition \ref{condition B} of Conjecture \ref{conjecture: Kotschick stronger} holds, i.e.\ there is a closed real 1-form $\alpha$ with $[\alpha]\in f^\ast H^1(A,\R)$ on $X$ without zeros.
Then we have
\[ \chi({}^{p}R^jf_{\ast}\K_X)=0, \]
for any $j\in \Z$ and any infinite field $\K$.
\end{proposition}

\begin{proof}
Up to perturbing $\alpha$ slightly, we may assume that $[\alpha]\in f^\ast H^1(A,\Q)$.
Multiplying by a suitbale integer, we thus reduce to the case where $[\alpha]\in f^\ast H^1(A,\Z)$.
Integration over $\alpha$ then yields as in \cite{Tischler} a submersive $\mathcal C^\infty$-map $g:X\to S^1$ to the circle with $g^\ast d\theta=\alpha$, where $\theta$ denotes the angular coordinate on $S^1$.
In particular, $g$ is a $\mathcal C^\infty$-fiber bundle with $g^\ast H^1(S^1,\R)\subseteq f^\ast H^1(A,\R)$.

Let $\K$ be any infinite field. We first mimic the proof of \cite[Theorem A.1]{HS} to produce a $\K$-local system on $X$ with no cohomology. Let $L_{\lambda}$ be a generic $\K$-local system on $S^1$ with monodromy given by $\lambda \in \K$. Set
\[ L=g^{\ast}L_{\lambda}.\]
Consider the Leray spectral sequence with $E_2$-term
\[ E^{p,q}_2=H^p(S^1,R^qg_{\ast}g^{\ast}L_{\lambda})=H^p(S^1,L_{\lambda}\otimes_{\K} R^qg_{\ast}\K_X)\Longrightarrow H^{p+q}(X,L).\]
Since $g$ is a $\cC^{\infty}$-fiber bundle, the sheaf $R^qg_{\ast}\K_X$ is a local system on $S^1$ with the stalk $V^q$ being a finite dimensional $\K$-vector space. Since $\K$ is infinite and so we can choose an element $\lambda\in \K$ such that $\lambda^{-1}$ is different from any of the eigenvalues of the natural monodromy operator on $V^q$ for all $q$. Therefore we conclude that
\[ H^0(S^1,L_{\lambda}\otimes_{\K} R^qg_{\ast}\K_X)=0, \quad \textrm{for all }q.\]
Since the Euler characteristic of any local system of finite rank on $S^1$ is zero, we also know that
\[ H^1(S^1,L_{\lambda}\otimes_{\K} R^qg_{\ast}\K_X)=0, \quad \textrm{for all }q.\]
Therefore, $E^{p,q}_2=0$ for all $p,q$ and we obtain $H^k(X,L)=0$ for all $k$.

Since $g^{\ast}H^1(S^1,\R) \subseteq f^{\ast}H^1(A,\R)$, 
the local system $L$ above is isomorphic to the pullback of some $\mathbb K$-local system on $A$.
By the semicontinuity of cohomology in families, we deduce that a generic  $\mathbb K$-local system $L_A$ of rank 1 on $A$ satisfies:
\begin{align}\label{eq:coho=0}
    H^i(X,f^\ast L_A)=0\ \ \ \ \text{for all $i$}.
\end{align}

We now set $L_X:=f^\ast L_A$ and consider the perverse Leray spectral sequence with $E_2$-term 
\begin{align*}
    E^{j,\ell}_2&=H^j(\Alb(X),{}^{p}R^{\ell}f_{\ast}L_X)=H^j(\Alb(X),{}^{p}R^{\ell}f_{\ast}(f^{\ast}L_A))\\
                &=H^j(\Alb(X),{}^pR^{\ell}f_{\ast}\K_{X}\otimes_{\K} L_A) \Longrightarrow H^{j+\ell}(X,L_X).
\end{align*}
Theorem \ref{thm: generic vanishing BSS} implies that $E^{j,\ell}_2=0$ for $j>0$, and so the spectral sequence degenerates at $E_2$-page. 
On the other hand, $H^{j+\ell}(X,L_X)=0$ for all $j,\ell$ by (\ref{eq:coho=0}) and so 
$$
 E^{j,\ell}_2=H^j(\Alb(X),{}^pR^{\ell}f_{\ast}\K_{X}\otimes_{\K} L_A)=0
$$
for all $j,\ell$.
Hence, 
\begin{align*}
    \chi(\Alb(X),{}^pR^{\ell}f_{\ast}\K_{X}\otimes L_A) =0, \quad \textrm{for all $\ell$}.
\end{align*} 
Using the invariance of Euler characteristic under twisting by a rank one local system, we conclude that 
\[ \chi(\Alb(X),{}^pR^jf_{\ast}\K_{X})=\chi(\Alb(X),{}^pR^jf_{\ast}\K_{X}\otimes_{\K} L_A)=0,\]
for all $j$ and any infinite field $\K$.\qedhere
\end{proof}

Before deriving a consequence of Proposition \ref{prop: vanishing euler characteristic for direct image of Albanese}, we recall 
\begin{definition}{\cite[Definition 4.1]{LMWBobadilla}}\label{definition: locally constant constructible complex}
Let $R$ be any commutative ring.   An $R$-constructible complex $\cF$ is \emph{locally constant} if the cohomology sheaves $\cH^j(\cF)$ are local systems for all $j$. \end{definition}

The following lemma is a version of \cite[Lemma 5.9]{LMWBobadilla}.
\begin{lemma}\label{lemma: criterion for locally freeness}
Let $\cF$ be a $\Z$-constructible complex on a complex manifold. If $\cF\otimes_{\Z}^L \K$ is locally constant for any infinite field $\K$, then $\cF$ is locally constant.
\end{lemma}

\begin{proof}
    The proof of \cite[Lemma 5.9]{LMWBobadilla} is reduced to \cite[Lemma 5.8]{LMWBobadilla}. But to show a morphism of bounded complexes of free $\Z$-modules being quasi-isomorphic, we just need to check it still holds after tensoring 
    with any infinite field (in fact one field per characteristic is enough). 
\end{proof}

\begin{corollary}\label{cor: Albanese morphism is a Z homology fiber bundle}
With the same assumption as in Proposition \ref{prop: vanishing euler characteristic for direct image of Albanese}, assume in addition that $A$ is simple. Then $f:X\to A$ is a $\Z$-homology fiber bundle. 
\end{corollary}

\begin{proof}
Note first that $f$ is a $\Z$-homology fiber bundle if and only if  
\[ \cH^j(Rf_{\ast}\Z_X)=R^jf_{\ast}\Z_{X}\]
are local systems for all $j$, i.e. the $\Z$-constructible complex $Rf_{\ast}\Z_{X}$ is locally constant.  
By Lemma \ref{lemma: criterion for locally freeness}, it suffices to show that for any infinite field $\K$, the $\K$-constructible sheaf 
\[Rf_{\ast}\Z_{X}\otimes_{\Z}^{L}\K=Rf_{\ast}\K_{X}\]
is locally constant. Then \cite[Proposition 4.3]{LMWBobadilla} says that we only need to check that the perverse cohomology sheaf
\[ {}^pR^jf_{\ast}\K_{X} \textrm{ is a shift of a local system for each $j$ and any infinite $\K$}.\]
Since $A$ is simple, by Corollary \ref{corollary: locally free criterion for non-simple perverse sheaf}, it suffices to show that
\[ \chi({}^pR^jf_{\ast}\K_{X})=0, \textrm{ for each $j$ and any infinite $\K$},\]
which follows from Proposition \ref{prop: vanishing euler characteristic for direct image of Albanese}. Therefore we conclude that $f:X\to A$ is a $\Z$-homology fiber bundle. 
\end{proof} 

\begin{proof}[Proof of Theroem \ref{thm: main}]
It suffices to prove \ref{condition B} $\Longrightarrow$ \ref{condition A}. Starting from \ref{condition B},  using Corollary \ref{cor: Albanese morphism is a Z homology fiber bundle} and the Bobadilla-Koll\'ar conjecture \ref{conjecture: BK} for the morphism $f:X\to A$, we deduce that $f:X \to A$ is smooth. Therefore any pullback of holomorphic 1-form has no zeros on $X$ and thus \ref{condition B} $\Longrightarrow$ \ref{condition A}.
\end{proof}

\begin{proof}[Proof of Corollary \ref{corollary: strong Kotschick conjecture}]
Let $f:X\to A$ be a proper map, where $A$ is a simple complex torus with $\dim A\geq \dim X-1$.
Assume that condition \ref{condition B} in Conjecture \ref{conjecture: Kotschick stronger} holds.
By Corollary \ref{cor: Albanese morphism is a Z homology fiber bundle}, $f$ is a $\Z$-homology fibration and in particular surjective.
The Bobadilla--Koll\'ar conjecture is clearly true for proper maps of relative dimension zero and it holds for proper maps of relative dimension one by \cite[Proposition 10]{Bobadilla}.
This concludes the argument because
  $\dim A\geq \dim X-1$.
\end{proof}

\begin{proof}[Proof of Corollary \ref{corollary: Kotschick conjecture large betti number}]
This is a direct consequence of Corollary \ref{corollary: strong Kotschick conjecture}.
\end{proof}

\appendix   

\section{Degenerate perverse sheaves on complex tori}

In this appendix, we provide a proof of Theorem \ref{thm: locally free criterion for perverse sheaves}.  

First, we adapt \cite[\S 2]{FK} to the analytic setting, where one studys the generic degree of a meromorphic map to a Grassmannian (e.g. the Gauss map). Let $M$ be a connected $k$-dimensional complex manifold. Let $V$ be an $n$-dimensional complex vector space, and let 
\[f : M \dashrightarrow G(k,V) \]
be a meromorphic map to the Grassmannian of $k$-dimensional linear subspaces of $V$. 
Up to replacing $M$ by a dense Zariski open subset, we can assume that $f$ is regular. 
Consider the flag variety
$F(k, n-1, V)$ and its projections
\[ \begin{tikzcd}
    {} &F (k, n-1, V )\arrow[dl,"p"] \arrow[dr,"q"]& {}\\
    G(k, V) & & G(n-1, V ) .
\end{tikzcd}
\]
Consider the fiber product diagram
\[ \begin{tikzcd}
\tilde{M}\colonequals M\times_{G(k, V)}F(k, n-1, V ) \arrow[r] \arrow[d] & F(k,n-1,V)\arrow[d,"p"]\\
M \arrow[r,"f"] & G(k,V)
\end{tikzcd}
\]
Since $p$ is a smooth map with $(n-k-1)$-dimensional fibers, $\tilde{M}$ is a connected $(n-1)$-dimensional complex manifold. Set-theoretically, we have
\[ \tilde{M}=\{ (x,V_0,W)\in M\times F(k,n-1,V) \mid f(x)=V_0, V_0\subseteq W \}.\]
The map $q$ induces a holomorphic map between complex manifolds of the same dimension
\begin{align*}
    q': &\tilde{M} \to F(k,n-1,V) \xrightarrow{q} G(n-1,V),\\
       &(x,V_0,W) \mapsto (V_0,W) \mapsto W.
\end{align*} 
Note that we have $f(x)\in G(k,W)$.
\begin{definition} 
We define $\deg f$ to be the degree of the map $q'$.
\end{definition}

With the set-up above, it is direct to deduce the following
\begin{proposition}\label{proposition: degree of map to Grassmannian}
If $W \subseteq V$ is a generic hyperplane, then $\deg f$ is equal to the number of $x \in M$ such that $f(x) \in G(k, W) \subseteq G(k, V)$. 
If $\deg f>0$, then for any such $x$, the map $f$ is locally an embedding near $x$ and the intersection $f(M) \cap G(k, W )$ is transversal at $x$.
\end{proposition}

Now we apply the discussion above to the Gauss map associated to a complex torus. We find it is easier to work in a more general setting. Let $G$ be a complex Lie group, let $\mathfrak{g}$ be its Lie algebra. If $Z \subseteq  G$ is an irreducible $k$-dimensional closed analytic subset, we have the meromorphic map (called Gauss map) 
\[ \Gamma_Z : Z \dashrightarrow G(k, \mathfrak{g}) \]
defined as follows. For $x \in G$, let
\[ \ell_x : G \to G, \quad  \ell_x (y) = xy \]
be the left multiplication by $x$. Then, for a smooth point $z \in Z$, we have $\Gamma_Z(z)=z^{-1}(T_zZ)$, which is the image of the differential map $\ell_{z^{-1}}$
\[ d_z\ell_{z^{-1}} : T_zZ \to T_eG = \mathfrak{g}.\]
Let $\Lambda_Z\subseteq T^{\ast}G$ denote the conic Lagrangian variety associated to $Z$, which is the closure in $T^{\ast}G$ of the conormal bundle in $G$ to the smooth locus of $Z$. For $\gamma \in \mathfrak{g}^{\ast}$, let $\Omega_{\gamma} \subseteq T^{\ast}G$ be the graph of the left invariant $1$-form $\omega_{\gamma}$ on $G$ associated to $\gamma$. 

\begin{proposition}\label{proposition: intersection number is the degree of Gauss map}
Let $\gamma \in \mathfrak{g}^{\ast}$ be a generic linear functional. Then $\Lambda_Z\cap \Omega_{\gamma}$ consists of finitely many points that are smooth on $\Lambda_Z$ and in which the intersection is transverse. Moreover,
\[ \deg \Gamma_Z=\# \abs{\Lambda_Z\cap \Omega_{\gamma}}.\]
\end{proposition}

\begin{proof}
    Apply Proposition \ref{proposition: degree of map to Grassmannian} where $M=Z, V=\fg, f=\Gamma_Z$. We can view $\gamma\in \fg^{\ast}$ as a generic hyperplane in $\fg$ and $\Gamma_Z(z)\subseteq \gamma\subseteq \fg$ if and only if $(z,\omega_{\gamma}(z))\in \Lambda_Z \cap \Omega_{\gamma}$.
\end{proof}

Finally, we can prove Theorem \ref{thm: locally free criterion for perverse sheaves}.
\begin{proposition}\label{Proposition A}
Let $A$ be a complex torus, which is simple as a torus. Let $\K$ be any field and let $\cP$ be a simple $\K$-perverse sheaf on $A$. If the Euler characteristic of $\cP$ vanishes, i.e.
    \[ \chi(\cP)=\sum_i (-1)^i \dim_{\K} H^i(A,\cP)=0,\]
    then $\cP$ is a shift of a local system.
\end{proposition}

\begin{proof}
    We adapt the proof of \cite[Proposition 10.1]{KW}. For $Z\subseteq A$ closed and irreducible, let $\Lambda_Z\subseteq T^{\ast}A$ denote the closure in $T^{\ast}A$ of the conormal bundle in $A$ to the smooth locus of $Z$.   
    By \cite[Definition 3.34]{MS}, the characteristic cycle associated to the $\K$-perverse sheaf $\cP$ on the complex manifold $A$ is a finite formal sum
    \[ CC(\cP)= \sum_{Z\subseteq A} n_Z\cdot \Lambda_Z, \quad \textrm{ with } n_Z \in \Z,\]
    where $Z$ runs through all closed irreducible subsets of $A$, 
    \[ n_Z\colonequals (-1)^{\dim Z} \cdot \chi(\mathrm{NMD}(\cP, Z)),\]
    and $\mathrm{NMD}(\cP, Z)$ is the sheaf-theoretic counterpart of the normal Morse data defined in \cite[\S 3.1, (26)]{MS}. By \cite[Example 3.26]{MS}, we have $n_Z
    \geq 0$ for $\K$-perverse sheaves. The Dubson-Kashiwara's microlocal index formula still holds in this setting: apply \cite[Theorem 3.38]{MS} when $f$ is a constant function. 
    Therefore we have
    \[ \chi(\cP) = \sum_{Z\subseteq A} n_Z \cdot d_Z, \]
    where
    \[ d_Z =\langle [\Lambda_A]\cdot [\Lambda_Z] \rangle_{T^{\ast}A} \in \N.\]
    Therefore, $d_Z$ can be computed as the intersection number of $\Lambda_Z$ with $\Omega_{\gamma}$ inside $T^{\ast}A$, where $\Omega_{\gamma}$ is the graph of a generic differential one-form $\gamma$ on $A$.    Since $A$ is a torus, the cotangent bundle $T^{\ast}A=A\times \C^g$ is trivial of rank $g=\dim A$, with two projection maps
    \[ \begin{tikzcd}
    {} &T^\ast A\arrow[dl,"\mathrm{pr}_1"] \arrow[dr,"\mathrm{pr}_2"]& {}\\
    A & & \C^g
\end{tikzcd}
\]
    Projecting from $\Lambda_Z\subseteq T^{\ast}A$ onto the second factor $\C^g$ induces a map
    \[ p: \Lambda_Z \to \C^g.\]
    It is easy to see that degree of $p$ is equal to the degree of the Gauss map associated to $Z$ and $A$ as discussed above.
    By Proposition \ref{proposition: intersection number is the degree of Gauss map}, the intersection number $d_Z=\# \abs{\Lambda_Z\cap \Omega_{\gamma}}$ is the generic degree of the Gauss map, which is equal to the generic degree of $p$.
    
    Since $n_Z,d_Z$ are both nonnegative, the assumption $\chi(P)=0$ implies that $d_Z=0$. Therefore $p$ is not surjective and $\dim p(\Lambda_Z)<g$. Then for some cotangential vector $\omega\in p(\Lambda_Z)$, the fiber $p^{-1}(\omega)$ is positive-dimensional. If $Z\neq A$, we can assume $\omega \neq 0$. Let $Y\subseteq A$ be the image of $p^{-1}(\omega)\subseteq T^{\ast}A$ under the map $T^{\ast}A\to A$. Then $\dim Y>0$, and up to a translation we can assume $0\in Y$. By construction, $\omega$ is normal to $Y$ in every smooth point of $Y$, so the preimage of $Y$ under the universal covering $\C^g \to A=\C^g/\Lambda$ lies in the hyperplane of $\C^g$ orthogonal to $\omega$. Thus the subtorus of $A$ generated by $Y$ is strictly contained in $A$ but non-zero, contradicting the assumption that $A$ is simple. 
    Therefore the characteristic cycle $\mathrm{CC}(\cP)$ only contains the zero section of $T^{\ast}A$ and hence $\cP$ is a shift of a local system, see e.g.\ Lemma 5.14 in the first arXiv version of \cite{LMWBobadilla}. 
\end{proof}

\begin{remark}
The proof of Proposition \ref{Proposition A} works verbatim for arbitrary possibly non-simple perverse sheaves. We presented an alternative argument to reduce to the non-simple case in Corollary \ref{corollary: locally free criterion for non-simple perverse sheaf} above for the convenience of the reader.
\end{remark}

\bibliographystyle{abbrv}
\bibliography{Zeros_of_one_forms}{}

\vspace{2cm}

\footnotesize{
\textsc{Institute of Algebraic Geometry, Leibniz University Hannover, Welfengarten 1, 30167 Hannover, Germany.} \\
\indent \textit{E-mail address:} \href{mailto:schreieder@math.uni-hannover.de}{schreieder@math.uni-hannover.de}

\vspace{\baselineskip}

\textsc{Max-Planck-Institut f\"ur Mathematik, Vivatsgasse 7, 53111 Bonn, Germany} \\
\indent \textit{E-mail address:} \href{mailto:ruijie.yang@hu-berlin.de}{ruijie.yang@hu-berlin.de}
}
\end{document}